\documentclass[12pt]{article}
\usepackage{amsthm}
\usepackage{amsmath,amscd}
\usepackage{amsfonts}
\usepackage[dvips]{graphicx}

\newtheorem{theorem}{Theorem}[section]
\newtheorem{lemma}{Lemma}[section]

\newtheorem{definition}{Definition}[section]
\newtheorem{corollary}{Corollary}[section]

\newtheorem{acknowledgment*}{Acknowledgment}
\newtheorem{remark}{Remark}[section]
\numberwithin{equation}{section}

\newcommand{\be}{\begin{equation}}
\newcommand{\ee}{\end{equation}}
\newcommand{\bd}{\begin{displaymath}}
\newcommand{\ed}{\end{displaymath}}

\newcommand{\ccM}{{\cal B}^+_M}
\newcommand{\ccn}{{\cal B}^{+,n}}

\newcommand{\R}{\mathbb R}

\newcommand{\z}{\zeta}

\begin{document}
\Large
\begin{center} \bf{Limit Theorems for Optimal Mass Transportation and Applications to  Networks}\end{center}
\normalsize
\begin{center} Gershon Wolansky\footnote{ gershonw@math.technion.ac.il}\\
Department of Mathematics \\
Technion, Israel Institute of Technology \\
Haifa 32000, Israel \end{center}
\begin{abstract}
It is shown that optimal network plans can be obtained as a limit of point allocations. These  problems are obtained by minimizing the mass transportation on the set of atomic measures of prescribed number of atoms.
\end{abstract}

\section{Introduction}
Optimal mass transportation was introduced by Monge some 200 years ago and
is, today, a source of a large number of results in analysis, geometry and
convexity.
\par
Optimal Network Theory was recently developed. It can be formulated in terms of
Monge-transport corresponding to some non-standard metrics.
 For updated references on optimal networks via mass transportation see [BS, BCM].
\par
In this paper we restrict ourselves to the transport of a finite number of points. Consider $N$ points $\{ x_1, \ldots x_N\}$ (sources) in a state
space (say, $\R^k$), and another $N$ points $\{ y_1, \ldots y_N\}\subset\R^k$ (sinks). For each source $x_i$ we attribute a certain amount of mass $m_i\geq 0$. Similarly, $m^*_i\geq 0$ is the capacity attributed to the sink $y_i$, while
$$ \sum_1^N m_i=\sum_1^N m_i^* > 0 \ . $$ We denote this system by an atomic measure
$\lambda:= \lambda^+-\lambda^-$ where
\be\label{lambda} \lambda^+=\sum_1^N m_i\delta_{x_i} \ \ \ ; \ \ \ \lambda^-=\sum_1^N m^*_i\delta_{y_i} \ee
where $\delta_{(\cdot)}$ is the Dirac delta function.
\par
The object is to transport the masses from the sources to sinks in an optimal way, such that the sinks are filled up according to their capacity. A natural cost was suggested by Xia [X]: For each $q>1$,   $ \widehat{W}^{(q)}(\lambda)$ is defined below:
\begin{definition} \label{gammal}
Given $\lambda$ as in (\ref{lambda}),
\begin{enumerate}
\item
  An oriented, weighted graph $(\hat{\gamma}, m)$ associated with $\lambda$ is a graph $\hat{\gamma}$ embedded in $\R^k$,  composed of vertices $V(\hat{\gamma})$ and edges $E(\hat{\gamma})$. The \em{orientation} of an edge $e\in E(\hat{\gamma})$ is determined by $\partial e= v_e^+-v_e^-$ where $v_e^\pm\in V(\hat{\gamma})$ are the vertices composing the end points of $e$. The graph $\hat{\gamma}$ and the capacity function $m:E(\hat{\gamma})\rightarrow \R^+\cup\{0\}$ satisfy
  \begin{description}
  \item{(a)}  $\{x_1, \ldots x_N, y_1, \ldots y_N\}\subset V(\hat{\gamma})$.
 \item{(b)} For each $i\in\{1, N\}$, $\sum_{\{e, x_i\in \partial^+e\}} m_e = m_i$ and $\sum_{\{e, y_i\in \partial^-e\}} m_e = m^*_i$, where $\partial^\pm e:= v_e^\pm$.
     \item{(c)} For each $v\in V(\hat{\gamma})- \{ x_1, \ldots y_N\}$, $\sum_{\{ e; v\in\partial^+ e\}}m_e = \sum_{\{ e; v\in\partial^- e\}}m_e$.
 \end{description}
 \item The set of all weighted graphs associated with $\lambda$ is denoted by $\Gamma(\lambda)$.
 \item
 \be\label{defgraph} \widehat{W}^{(q)}(\lambda):= \inf_{(\hat{\gamma},m)\in  \Gamma(\lambda)} \sum_{e\in E(\hat{\gamma})} |e|m_e^{1/q}\ee
 \end{enumerate}
\end{definition}

There are two special cases which should be noted. In the limit case $q=1$ the optimal graph satisfies $V(\hat{\gamma})= \{ x_1, \ldots y_N\}$ and $\widehat{W}^{(1)}(\lambda)=W_1(\lambda^+, \lambda^-)$.
Here $W_q(\lambda^+, \lambda^-)$ for $q\geq 1$ is the Wasserstein distance between $\lambda^+$ to $\lambda^-$,
$$ W_q(\lambda^+, \lambda^-):= \left(\min_{ \{\gamma^{i,j}\} }\sum_1^N\sum_1^N |x_i-y_j|^q \gamma_{i,j}\right)^{1/q} \ , $$
the minimum is taken in the set of $N\times N$ matrices satisfying
$$ \gamma_{i,j}\geq 0 \ \ \ , \ \ \sum_{i=1}^N \gamma_{i,j}=m^*_j \ \ \ \ ; \ \ \ \sum_{j=1}^N \gamma_{i,j}=m_i \ . $$
In particular, $W_1$ depends only on the difference $\lambda=\lambda^+-\lambda^-$ (which is not the case for $q>1$). \par

The second case is the limit $q=\infty$. This is the celebrated {\it Steiner Tree Problem} [HRW]:
$$   \inf_{\hat{\gamma}\in  \Gamma(\lambda)} \sum_{e\in E(\hat{\gamma})} |e| \ , $$
where, this time, $\Gamma(\lambda)$ is the set of all graphs satisfying $\{x_i, y_j; m_i, m^*_j>0\}\subset V(\hat{\gamma})$ and is, actually, independent of the masses $m_i$ and capacities $m^*_i$ (assumed positive).
\par
In [W, Thm 2]  it was shown  that $W_1$ is obtained from $W_q$ by an asymptotic expression for the limit of infinite mass:
\begin{theorem} If $\lambda=\lambda^+-\lambda^-$ is any Borel measure satisfying $\int d\lambda=0$, then
$$\lim_{M\rightarrow\infty}  M^{1-1/q} \min_{\mu\in\ccM}W_q\left(\mu+\lambda^+, \mu+\lambda^-\right)= W_1(\lambda^+, \lambda^-)$$
where $\ccM$ stands for the set of all positive Borel measures $\mu$  normalized by  $\int d\mu=M$.
\end{theorem}
If, in particular, $\lambda$ is an atomic measure of the form (\ref{lambda}), than it can be shown that for fixed $M$ the minimizer of $W_q\left(\mu+\lambda^+, \mu+\lambda^-\right)$ in $\ccM$ is an atomic measure of a finite number of atoms as well. \par
 The main result of the current paper demonstrates that the network cost $\widehat{W}^{(q)}$ is obtained by similar expression, where  the total mass $M$ is replaced by the {\it cardinality of the support of the atomic measure} $\mu$.

\section{Main results}

For each $n\in \mathbb{N}$, let $\ccn$ be the set of all atomic, positive measures of at most $n$ atoms, that is:
$$\ccn:=\left\{  \sum_{j=1}^n \alpha_j \delta_{z_j} \ \ \ ; \ \ z_j\in \R^k, \ \alpha_j\geq 0  \ \right\} . $$
\begin{theorem}\label{th1}
For any $q>1$ and $\lambda$ as in (\ref{lambda})
\be\label{main} \lim_{n\rightarrow\infty}  n^{1-1/q} \inf_{\mu\in\ccn}W_q\left(\mu+\lambda^+, \mu+\lambda^-\right)= \widehat{W}^{(q)}(\lambda) \ . \ee
\end{theorem}
The set $\ccn$ is, evidently, not a compact one. Still we claim
\begin{lemma}\label{overWlem} For each $n\in \mathbb{N}$, a minimizer
 $\mu_n\in\ccn$
 \be\label{overW}\overline{W}^{(n)}_q(\lambda^+, \lambda^-):= \inf_{\mu\in\ccn}W_q\left(\mu+\lambda^+, \mu+\lambda^-\right)=W_q\left(\mu_n+\lambda^+, \mu_n+\lambda^-\right) \ee  exists.
 \end{lemma}
 \begin{remark}
 Note that $\overline{W}^{(n)}_q$ depends on each of the component $\lambda^\pm$ while the limit $\widehat{W}^{(q)}=\lim_{n\rightarrow\infty}  n^{1-1/q} \overline{W}^{(n)}_q$ depends only on the difference $\lambda=\lambda^+-\lambda^-$.
 \end{remark}
\begin{theorem}\label{th2}
Let $\mu_n$ be a regular\footnote{see Definition~\ref{regular}} minimizer of
$W_q\left(\mu+\lambda^+, \mu+\lambda^-\right)$
in $\ccn$.
Then   the associated optimal plan spans a reduced weighted tree\footnote{ See Definitions~\ref{defgraphdef} and \ref{defred}} $(\hat{\gamma}_n, m_n)$
 which converges (in Hausdorff metric) to an optimal graph $(\hat{\gamma},m)\in \Gamma(\lambda)$ of (\ref{defgraph}) as $n\rightarrow\infty$, \end{theorem}

\section{Auxiliary results}
We first reformulate  $\overline{W}^{(n)}_q$, as given by (\ref{overW}), in terms of a linear programming:
\par
  Given $q>1$, $n\in \mathbb{N}$,
 $Z=(z_{N+1}, \ldots z_{N+n})\in (\R^k)^n$, $\lambda=\lambda^+-\lambda^-$ as given by (\ref{lambda}) and  $\gamma:= \{\gamma_{i,j} \ 1\leq i,j \leq n+N\}\in \Gamma(n,\lambda^+, \lambda^-):=$
 \begin{multline}\label{linpro} \left\{ \gamma_{i,j}\geq 0 \ , \ \ \ 1\leq k\leq N  \Longrightarrow\sum_{i=1}^{n+N} \gamma_{k,i}= m_k, \ \ \ \ \sum_{i=1}^{n+N} \gamma_{i,k}= m^*_k \ \ \right.
 \\
 \left.
\sum_{i=1}^{n+N}\gamma_{i,j}=\sum_{i=1}^{n+N}\gamma_{j,i}  \ \text{for any} \ \ N+1\leq j \leq n+N  \right\} \ , \end{multline}
Let
\be\label{Fq} F_q(Z, \gamma):= \sum_{1}^{n+N}\sum_{1}^{n+N} \gamma_{i,j}F_{i,j}(Z)\ee where

\be\label{Fqq} F_{i,j}(Z):=\left\{
\begin{array}{cc}
  |z_i-z_j|^q & N+1\leq i,j\leq n+N  \\
  |x_i-z_j|^q  & 1\leq i\leq N, N+1\leq j\leq n+N \\
    |z_i-y_j|^q & 1\leq j\leq N, N+1\leq i\leq n+N  \\
  |x_i-y_j|^q & 1\leq i,j\leq N
\end{array} \right.
\ee

We observe
\be\label{discon}\overline{W}^{(n)}_q(\lambda^+, \lambda^-)\equiv \inf_{Z\in (\R^k)^n, \gamma\in\Gamma(n,\lambda^+, \lambda^-)} F_q(Z, \gamma)  \ . \ee

Our first object is to prove Lemma~\ref{overWlem}, that is, to replace the "$\inf$" in (\ref{discon}) by "$\min$".

\begin{definition}\label{regular}
 $\gamma\in \Gamma(n,\lambda^+, \lambda^-)$
 is called a {\em regular plan}
if it satisfies the following for any $1\leq i,j\leq n+N$:
\begin{description}
\item{(a)} if $k\geq 1$ and $i_1=i, i_2, \ldots i_k=i$ then
$\Pi_{l=1}^{k-1}\gamma_{i_l, i_{l+1}}=0$. (In particular
$\gamma_{i,j}\gamma_{j,i}=0$ and  $\gamma_{i,i}=0$).
\item{(b)} If $k>1$, $k^{'}\geq 1$ and $\{i_1=i, i_2, \ldots i_k=j\}\not\equiv \{i^{'}_1=i, i^{'}_2, \ldots i^{'}_{k^{'}}=j\}$
then  \\ $\left(\Pi_{l=1}^{k-1}\gamma_{i_l, i_{l+1}}\right)\left(\Pi_{l=1}^{k^{'}-1}
\gamma_{i^{'}_l, i^{'}_{l+1}}\right)=0$.
   \end{description}
If $\gamma$ is a regular plan, then $\mu\in \ccn$ is called a {\em regular measure} if for each $i\in\{N+1,\ldots n+N\}$ there exists $z_i\in\R^k$ where $\mu(\{z_i\})=\sum_{j=1}^{n+N}\gamma_{i,j}$.
\end{definition}

\begin{lemma}
For each $Z\in (\R^k)^n$ and any plan $\gamma\in\Gamma(n,\lambda^+, \lambda^-)$ there exists a regular  plan $\gamma^r\in \Gamma(n,\lambda^+, \lambda^-)$ satisfying $F_q(Z,\gamma^r; )\leq F_q(Z,\gamma)$.
\end{lemma}
\begin{proof}
\begin{description}
\item{a)} Assume $\Pi_{l=1}^{k-1}\gamma_{i_l, i_{l+1}}>0$. Let $i_{l_0}$ such that $\gamma_{i_{l_0}, i_{l_0}+1}\leq\gamma_{i_l, i_{l}+1}$ for any $1\leq l<k$. Then
     $\gamma^{r_1}_{i_l, i_{l}+1}:= \gamma_{i_l, i_{l}+1}-\gamma_{i_{l_0}, i_{l_0}+1}$ while
     $\gamma^{r_1}_{i,j}=\gamma_{i,j}$ otherwise. It follows that $\gamma^{r_1}\in\Gamma(n,\lambda^+, \lambda^-)$ and $F_q(Z,\gamma^{r_1})\leq F_q(Z,\gamma)$. Thus $\gamma^{r_1}$ verifies (a).
     \item{b)} We  may assume that $\{i_2, \ldots i_{k-1}\}\cap \{ i^{'}_2, \ldots i^{'}_{k^{'}-1}\}=\emptyset$ for, otherwise, choose 2 pairs of indices $i_l=i^{'}_{l^{'}}$ and $i_m=i^{'}_{m^{'}}$ for which 
         $\{i_{l+1}, \ldots i_{m-1}\}\cap \{ i^{'}_{l^{'}+1}, \ldots i^{'}_{m^{'}-1}\}=\emptyset$. 
         
         Assume $\left(\Pi_{l=1}^{k-1}\gamma^{r_1}_{i_l, i_{l+1}}\right)\left(\Pi_{l=1}^{k^{'}-1}
\gamma^{r_1}_{i^{'}_l, i^{'}_{l+1}}\right)>0$. Assume (with no limitation to generality) that $\sum_{l=1}^{k-1}|Z_{i_l}-Z_{i_{l+1}}|^{1/q} \geq \sum_{l=1}^{k^{'}-1}|Z_{i^{'}_l}-Z_{i^{'}_{l+1}}|^{1/q}$.  Let $i_{l_0}$ such that $\gamma^{r_1}_{i_{l_0}, i_{l_0}+1}\leq\gamma^{r_1}_{i_l, i_{l}+1}$ for any $1\leq l<k$. Then set
$$ \gamma^{r}_{i^{'}_l, i^{'}_{l+1}}:= \gamma^{r_1}_{i^{'}_l, i^{'}_{l+1}}+\gamma^{r_1}_{i_{l_0}, i_{l_0}+1} $$
$$ \gamma^{r}_{i_l, i_{l+1}}:= \gamma^{r_1}_{i_l, i_{l+1}}-\gamma^{r_1}_{i_{l_0}, i_{l_0}+1} $$
while $\gamma^{r}_{i, j}= \gamma^{r_1}_{i, j}$ otherwise. Then $\gamma^{r}$ verifies (\ref{linpro}) while
\begin{multline}F_q(Z,\gamma^{r})= F_q(Z,\gamma^{r_1})- \gamma^{r_1}_{i_{l_0}, i_{l_0}+1}\left[ \sum_{l=1}^{k-1}|Z_{i_l}-Z_{i_{l+1}}|^{1/q}- \sum_{l=1}^{k^{'}-1}|Z_{i^{'}_l}-Z_{i^{'}_{l+1}}|^{1/q}\right] \\ \leq F_q(Z,\gamma^{r_1})   \leq F_q(Z,\gamma)  \end{multline}

\end{description}

\end{proof}
  \begin{lemma}\label{compact} The set of regular plans in $\ccn$ associated with $\Gamma(n, \lambda^+, \lambda^-)$ (\ref{Fq}) is compact.
  \end{lemma}
  \begin{proof}
   Let $z_i$ be some point in the support of $\mu$ where $\mu(\{z_i\})=Q$. We show an a-priori bound on $Q$ (hence compactness). By (\ref{linpro}) there exists a point $z_{i_2}$ where $\gamma_{i, i_2}\geq Q/(N+n)$.  We can define such a chain $i=i_1, i_2, \ldots$ where $\gamma_{i_l, i_{l+1}}> \mu(\{z_{i_l}\})/(n+N)$. In particular it follows that $\mu(\{z_{i_l}\}) \geq Q/(n+N)^{l-1}$.
  By part (a) of the definition of regular plans, this chain must be of length ar most $n$. By (\ref{linpro}) it must  end at some $i_k:= j\in \{1, \ldots N\}$.     So, $\mu(\{ z_j\}) \geq   Q/(n+N)^{k-1} \geq  Q/(n+N)^{n-1}$. On the other hand, $\mu(\{z_j\}) \leq \max_{1\leq l\leq N} \max\{ m_l, m^*_l\}  := M$ so  $Q\leq M(n+N)^{n-1}$.
  \end{proof}
  \begin{corollary}\label{cor1}
  For  fixed $Z\in (\R^k)^n$, $\lambda$ satisfying (\ref{lambda}) and $q>1$,  the function $F_q$ admits a minimizer $\gamma\in \Gamma(n,\lambda^+, \lambda^-)$. Moreover, this minimizer is regular.
  \end{corollary}
  \begin{proof} of lemma~\ref{overWlem}

  For a fixed  $\lambda$ satisfying (\ref{lambda}) and $q>1$ it follows from Corollary~\ref{cor1} that
  $$ \overline{F}_q(Z, \gamma):= \inf_{\gamma\in\Gamma(n,\lambda^+, \lambda^-)}F_q(Z,\gamma)
   = \min_{\gamma\in\Gamma(n,\lambda^+, \lambda^-)}F_q(Z,\gamma) \ . $$
  It is also evident that $\overline{F}_q$ is continuous and coercive on  $(\R^k)^n$ and that
 $\overline{W}^{(n)}_q(\lambda^+, \lambda^-)=\min_{Z\in (\R^k)^n}\overline{F}_q(Z,\lambda)$.
 In particular (\ref{discon})
   is attained at a pair $(Z,\gamma)$ where
  \be\label{mun} \mu_n:=\sum_{i,j=N+1}^{n+N}\gamma_{i,j}\delta_{z_i}\in \ccn\ee
  is a regular minimizer of (\ref{overW}).
 \end{proof}
Next we associate a weighted graph $(\hat{\gamma},m)$ with a transport plan $\gamma\in
\Gamma(n,\lambda^+, \lambda^-)$ and $Z\in (\R^k)^n$ as follows (see Fig 1)
\begin{definition}\label{defgraphdef}
Let $Z=\{ z_{N+1}, \ldots z_{N+n}\}\in (\R^k)^n$ and  $\gamma\in \Gamma(n,\lambda^+, \lambda^-)$.    The associated  weighed, directed graph $(\hat{\gamma},m)$ is defined as
\par\noindent
\begin{description}
 \item{(i)} $V(\hat{\gamma})=\{  x_1, \ldots y_N, z_{1}, \ldots z_{n}\}:= \{ \zeta_1, \ldots \zeta_{n+2N}\}$.
      \item{(ii)} $E(\hat{\gamma})$ is given by the set of segments
$e_{k, l}:= [\zeta_k, \zeta_l]$ for which $\gamma_{k,l}>0$, while $\partial e_{k, l}=\zeta_k-\zeta_l$.
\item{(iii)}  $m_{e_{k,l}}:= \gamma_{k.l}$.
\item{(iv)} For each  $z_i\in  V(\hat{\gamma})$, $deg(z_i):= \#\{ j;  \gamma_{i,j}+\gamma_{j,i}>0\}$.
\end{description}
\end{definition}
\begin{figure}
\includegraphics[height=5cm,width =20cm]{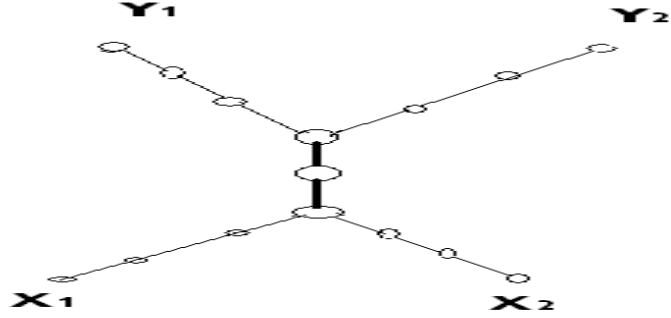}
\caption{A tree associated with a regular transport plan ($N=2$, $n=11$}
\end{figure}
\begin{lemma}\label{lemreg}
Let  $(Z,\gamma)$ as in Definition~\ref{defgraphdef} where $\gamma$ is a regular plan in $\Gamma(\lambda,n)$. Then the associated graph $(\hat{\gamma},m)$ contains no cycles. In addition,
$|E(\hat{\gamma})|\leq n+2N^3$.
\end{lemma}
\begin{proof}
The result that the graph $\hat{\gamma}$ contains no cycles follows directly from Definition~\ref{regular}-a. \par
It follows that any vertex $v\in V(\hat{\gamma})$ must belong to a  chain
$C_{i, j}:= \z_1, \ldots \z_k$ where $k\leq n$,  $\z_1=x_i$ and $\z_k=y_j$. By Definition~\ref{regular}-b there exists at most one such a chain for any pair $(x_i, y_j)\in \{x_1, \ldots x_N\}\times \{y_1, \ldots y_N\}$. In particular there exists at most $N^2$ such chains.
\par
Let now $\z_l\in C_{i, j}$. If the degree of $\z_l$ is greater than $2$, there exist $deg(\z_l)-1>1$  chains  which contain $\z_l$. By Definition~\ref{regular}-b it follows that if two chains $C_{i^{'}, j^{'}}$, $C_{i^{"}, j^{"}}$ intersect the chain $C_{i, j}$ then either
$C_{i^{'}, j^{'}}=C_{i^{"}, j^{"}}$ (and, in particular, they intersect $C_{i,j}$ at the same point), or
 $i^{"}\not= i^{'}$ and $j^{"}\not= j^{'}$. Hence the number of chains crossing $C_{i,j}$ is bounded by $2N$. As the number of chains $\{ C_{i,j}\}$ is bounded by $N^2$ it follows that there exists at most  $2N^3$ chains which intersect other chains. Hence $\sum_{v\in V(\hat{\gamma})} (deg(v)-2) \leq 2N^3$ which implies the result.
\end{proof}
Next, we elaborate some properties of an optimal regular plan.
\begin{definition} A chain of a regular plan is a sequence of indices $i_1, \ldots, i_k$ such that $\gamma_{i_l, i_{l+1}}>0$ for $k>l\geq 1$ while $\gamma_{i_l, j}=0$ for any $j\in \{ 1 \ldots, n+2N\}$. A maximal chain  is a chain which is not contained in a larger chain.
\end{definition}
\begin{remark}\label{eq}By (\ref{linpro}) we also get that $\gamma_{i_l, i_{l+1}}$ is a constant along any maximal chain $i_1, \ldots, i_k$  where $1<l<k$.\end{remark}
\begin{lemma}\label{line}
If $\gamma$ is a regular  {\em optimal} plan then for any chain $\{\zeta_{i_1}, \ldots \zeta_{i_k}\}$,
$\z_{i_{l+1}}-\z_{i_l}=\z_{i_{l^{'}+1}}-\z_{i_{l^{'}}}$
for any $l, l^{'}\in \{1, \ldots k-1\}$. In particular, all points on a chain of the associated  directed graph corresponding to an {\it optimal} plan are equally spaced on a {\em line segment}.
\end{lemma}
\begin{proof}
If $\gamma^R_0$ is a regular optimal plan then $Z=(z_1, \ldots z_n)$ is a minimizer of $F_q(Z, \gamma_0^R)$ in $(\R^k)^n$. In particular $\frac{\partial F_q}{\partial z_j}=0$ holds for any $1\leq j\leq n$. If $j=i_l$ is embedded in a chain then by definition and Remark~\ref{eq} we obtain
$$ \frac{\partial F_q}{\partial z_j}=q\gamma_{i_l, i_{l+1}}  \left[ \frac{z_{i_l}-z_{i_{l-1}}}{|z_{i_l}-z_{i_{l-1}}|^{q-2}} - \frac{z_{i_{l+1}-z_{i_{l}}}}{|z_{i_{l+1}}-z_{i_{l}}|^{q-2}} \right]\ =0 $$
which implies the result.
\end{proof}
Let us now re-define the associated   directed graph $(\hat{\gamma},m)$ corresponding to {\em an optimal} regular plan
(see Fig 2)
\begin{definition}\label{defred}
The {\em reduced} weighted graph $(\hat{\gamma}_R,m)$ associated with an optimal regular plan is obtained from $(\hat{\gamma},m)$ (Definition~\ref{defgraphdef}) by identifying all edges corresponding to a maximal chain
$\{ i_1, \ldots i_{k}\}$ with a single edge $[\z_{i_1}, \z_{i_k}]$ and assigning the the common weight $m_e= \gamma_{i_l, i_{l+1}}$ to this edge (recall Remark~\ref{eq}).
\end{definition}
\begin{figure}
\includegraphics[height=5cm,width =20cm]{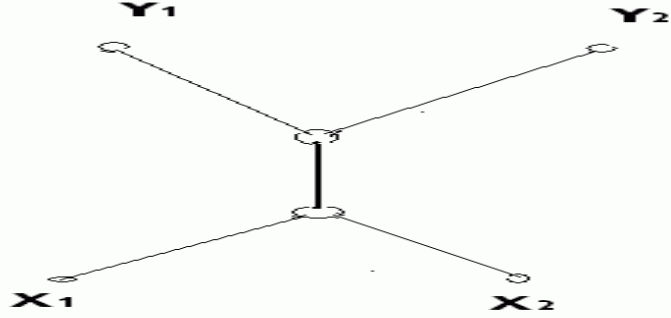}
\caption{The reduced version of the tree presented in Fig 1: All vertices of degree $2$ removed. }
\end{figure}
\begin{corollary}\label{redcor}
A reduced weighted graph $(\hat{\gamma}^n_R,m)$ associated with an optimal regular plane in $\ccn$
satisfies the following:
\begin{description}
\item{(i)} All the vertices of $\hat{\gamma}^n_R$ are of degree at least 3.
\item{(ii)} The number of vertices of $\hat{\gamma}^n_R$ is at most $2N^3$ where $N$ is the number of atoms of $\lambda^\pm$ (in particular, independent of $n$).
    \item{(iii)} All the edges of $\hat{\gamma}^n_R$ are line segments.
    \item{(iv)} There exists $C>0$, depending only on $N$, such that $C>m_e>1/C$ for any $e\in E(\gamma^n_R)$.
    \item{(v)} There is a compact set $K\subset \R^k$ which contains $\hat{\gamma}^n_R$ for any $n\in \mathbb{N}$.
\end{description}
\end{corollary}
\begin{proof}
Part (i) follows directly from Definition~\ref{defred}. Part (ii) from Lemma~\ref{lemreg}, part (iii) from Lemma~\ref{line}. To prove part (iv) we repeat the proof of Lemma~\ref{compact}, with the additional information of (ii) (that is, the bound on the number of edges is independent of $n$). Part (v) is evident.
\end{proof}
\section{Proof of Theorems~\ref{th1} and \ref{th2}}
\begin{proof} of theorem~\ref{th1}: \\
Let $(\hat{\gamma},m)$ be a weighted graph. Then by the H\"{o}lder inequality
 \be\label{ineq1}\sum_{e\in E(\hat{\gamma})}m^{1/q}_e|e|\leq \left(\sum_{e\in E(\hat{\gamma})}m_e|e|^q\right)^{1/q}|E(\hat{\gamma})|^{(q-1)/q} \ .  \ee
If, moreover, $(\hat{\gamma},m)$ is obtained from a regular plan $\gamma\in\Gamma(n,\lambda^+, \lambda^-)$ then
 \be\label{ineqWq} W^q_q(\mu_n+\lambda^+,\mu_n+\lambda^-)\leq \sum_{e\in E(\hat{\gamma})}m_e|e|^q\ee
 where $\mu_n\in \ccn$ associated with $\gamma$ via (\ref{mun}). By Lemma~\ref{overWlem} there exists an optimal measure $\mu_n\in \ccn$. Hence
 (\ref{ineqWq}) holds with an equality for this choice of $\mu_n$. Moreover, $\mu_n$ can be chosen to be a regular measure (Definition~\ref{regular})
 hence, by (\ref{ineq1},\ref{ineqWq}) and by Lemma~\ref{lemreg}
$$\widehat{W}_q(\lambda)\leq \sum_{e\in E(\hat{\gamma})}m^{1/q}_e|e|\leq \overline{W}^{(n)}_q(\lambda^+, \lambda^-)|n+2N^3|^{(q-1)/q} \ .  $$
This implies the inequality
$$\liminf_{n\rightarrow\infty}  n^{1-1/q} \overline{W}^{(n)}_q\left(\lambda^+, \lambda^-\right)\geq  \widehat{W}^{(q)}(\lambda) \ . $$
To prove the reverse inequality in (\ref{main}) we consider an optimal weighed graph $(\hat{\gamma},m)$ of $\widehat{W}_q(\lambda)$ and
 construct $\mu_n\in\ccn$ supported on $\hat{\gamma}$ which satisfy
  $$\lim_{n\rightarrow\infty}  n^{1-1/q} W_q\left(\mu_n+\lambda^+, \mu_n+\lambda^-\right)= \sum_{e\in E(\hat{\gamma})} m_e^{1/q} |e| = \widehat{W}^{(q)}(\lambda) \ . $$
  Assume $n_e$ is the number of points of $\mu_n$ on the edge $e$, and any atom of $\mu_n$ in $e$ is of weight $m_e$.  The contribution to $W_q^q(\mu_n+\lambda^+,  \mu_n+\lambda^-)$ from
$e$ is, then
$$ \approx m_e \left(\frac{|e| }{n_e}\right)^q n_e= \frac{m_e |e|^q}{n_e^{q-1}}$$
$$ n^{q-1}W_q^q( \mu_n+\lambda^+,\mu_n+\lambda^-) \approx n^{q-1}\sum_{e\in E(\hat{\gamma})}\frac{m_e |e|^q}{n_e^{q-1}}$$
The constraint on $n_e$ is given by
$\sum_{e\in E(\hat{\gamma})} n_e=n$.
Let us rescale $w_e:= n_e/n$. Then we need to minimize
$$ F(w):=\sum_{e\in E(\hat{\gamma})}\frac{m_e |e|^q}{w_e^{q-1}}$$
subjected to
$\sum_{e\in E(\hat{\gamma})} w_e=1$.
Let $\alpha$ be the Lagrange multiplier with respect to the constraint $\sum_{e\in E(\hat{\gamma})}w_e$. Since $F$ is convex in $w_e$ we get that $F$ is maximized at
\be\label{min1}\max_{\alpha}\min_w F(w)+ \alpha (\sum_{e\in E(\hat{\gamma})} w_e-1) \ . \ee
So. let
$$G(\alpha):= \min_w F(w)+  \sum_{e\in E(\hat{\gamma})} w_e\alpha$$
The minimizer is obtained at
$$ (q-1)\frac{m_e|e|^q}{w_e^q}=\alpha\Longrightarrow   w_e= (q-1)^{1/q}m_e^{1/q}|e| \alpha^{1/q} $$
so
$$ G(\alpha)= q(q-1)^{1/q-1} \sum_{e\in E(\hat{\gamma})}m_e^{1/q}|e|\alpha^{(q-1)/q} \ . $$
and the minimum is obtained at
 \be\label{111} \min_{(m,\hat{\gamma})\in \Gamma(\lambda)}\max_{\alpha} G(\alpha)-\alpha= \max_{\alpha} q(q-1)^{1/q-1} \widehat{W}^{(q)}(\lambda)\alpha^{(q-1)/q}-\alpha = \left(\widehat{W}^{(q)}(\lambda)\right)^q\ee
\end{proof}
\begin{proof} of Theorem~\ref{th2}: \\
Let us consider the sequence of reduced weighted graphs $(\hat{\gamma}^n_{R}, m_n)$ (see Definition~\ref{defred}) associated with a regular minimizer $\gamma_n$. By Corollary~~\ref{redcor}-(v) there exists a limit $\hat{\gamma}_R$  (in the sense of Hausdorff metric) of a subsequence of  $\hat{\gamma}^n_{R}$. By (ii-iv) of the Corollary,  $|E(\hat{\gamma}_R)| < 2N^3$ and is $E(\hat{\gamma}_R)$ is composed of lines. Moreover, the weights $m_n: E(\hat{\gamma}^n_R)\rightarrow \R^+$ converges also, along a subsequence, to
 $m: E(\hat{\gamma}_R)\rightarrow \R^+$ so $(m, \hat{\gamma}_R)\in \Gamma(\lambda)$ (see Definition~\ref{gammal}-(2)). Moreover
  \be\label{lim22}\lim_{n\rightarrow\infty} \sum_{e\in E(\hat{\gamma}_R^n)} m^{1/q}_{n,e}|e| =  \sum_{e\in E(\hat{\gamma}_R)} m^{1/q}_e|e| \ee
 By definition of the reduced graph (Definition~\ref{defred}) and, in particular, Remark~\ref{eq}  we observe that $ \sum_{e\in E(\hat{\gamma}_R^n)} m^{1/q}_{n,e}|e|$ is identical to the same expression on the non reduced graph $\hat{\gamma}^n$:
  \be\sum_{e\in E(\hat{\gamma}_R^n)} m^{1/q}_{n,e}|e| = \sum_{e\in E(\hat{\gamma}^n)} m^{1/q}_{n,e}|e| \ .  \ee
 However, on the non-reduced graphs we also have the inequalities (\ref{ineq1}, \ref{ineqWq})
 \be\label{lim33}\sum_{e\in E(\hat{\gamma}^n)}m^{1/q}_{n,e}|e|\leq \left(\sum_{e\in E(\hat{\gamma}^n)}m_{n,e}|e|^q\right)^{1/q}|E(\hat{\gamma}^n)|^{(q-1)/q}= \overline{W}^{(n)}_q(\lambda^+, \lambda^-)|E(\hat{\gamma}^n)|^{(q-1)/q}\ee
where $\overline{W}^{(n)}_q$ as defined in (\ref {overW}). Here we also used the optimality of $\gamma^n$.
\par
Finally, from Theorem~\ref{th1}
$$ \lim_{n\rightarrow\infty}\overline{W}^{(n)}_q(\lambda^+, \lambda^-)|E(\hat{\gamma}^n)|^{(q-1)/q}= \lim_{n\rightarrow\infty}\overline{W}^{(n)}_q(\lambda^+, \lambda^-)n^{(q-1)/q} = \widehat{W}_q(\lambda) \ . $$
 This and (\ref{lim22}-\ref{lim33}) yields
 $$ \sum_{e\in E(\hat{\gamma}_R)} m^{1/q}_e|e| \leq \widehat{W}_q(\lambda) $$
while the opposite inequality follows from the definition of $\widehat{W}_q$.
\end{proof}

\begin{center} {\bf References}\end{center}
\begin{description}
\item{[BS]} G. Buttazzo and E. Stepanov {\it Optimal Urban Networks via Mass Transportation}, Lec. Notes in Math. {\bf 1961}, Springer (2009)
\item{[HRW]} \ F.K. Hwang, D.S. Richards, P. Winter, {\it The Steiner Tree Problem}. Elsevier, North-Holland, 1992
\item{[BCM]} M. Bernot, V. Caselles and J.M. Morel, {\it Optimal Transportation Networks}, Lec. Notes in Math. {\bf 1955}, Springer (2009)
\item{[X]} Q. Xia \ \ {\it Optimal paths related to transport problems}, Communications in Contemporary Math, {\bf 5}, 251-279, 2003
\item{[W]} \ \ G. Wolansky, {\it Some new links between the weak KAM and Monge problems}, arXiv:0903.0145
\end{description}
\end{document}